\documentclass[article, 12pt]{amsart}
\usepackage{amscd}
\usepackage{latexsym}
\usepackage{amssymb, amsmath, amsthm}
\usepackage{graphicx}
\usepackage{subfigure}
\usepackage{url}

\newtheorem{theorem}{Theorem}[section]
\newtheorem{lemma}[theorem]{Lemma}
\theoremstyle{remark}
\newtheorem{remark}[theorem]{\bf Remark}
\newtheorem{corollary}[theorem]{\bf Corollary}

\newtheorem{proposition}[theorem]{\bf Proposition}

\newcommand{\n}{\mathbb{N}}

\newcommand{\f}{\mathcal{F}}

\begin{document}

\title[Ample generics]{The group of homeomorphisms of the Cantor set has ample generics}

\author{Aleksandra Kwiatkowska}


\begin{abstract}
 We show that the group of homeomorphisms of the Cantor set $H(2^\n)$ has ample generics, that is,  for every $m$ the diagonal conjugacy action 
$g\cdot(h_1,h_2,\ldots, h_m)=(gh_1g^{-1},gh_2g^{-1},\ldots, gh_mg^{-1})$
of $H(2^\n)$ on $H(2^\n)^m$ has a comeager orbit. This answers a question of Kechris and Rosendal.
We show that a generic tuple in $H(2^\n)^m$ can be taken to be  the limit of a certain projective Fra\"{i}ss\'{e} family. 
 We also give an  example of a projective Fra\"{i}ss\'{e} family, which has a simpler description than the one considered in the general case, and
such that its limit is a homeomorphism of the Cantor set that has a comeager conjugacy class.


\end{abstract}
\maketitle

\section{Introduction}\label{szero}

A group $G$ acts on itself by conjugation $g\cdot h=ghg^{-1}$. Orbits in this action are \emph{conjugacy classes}.
A classical result by Halmos asserts that the group of all measure preserving transformations of the standard Lebesgue space
has a dense conjugacy class; his proof uses the fundamental lemma due to Rokhlin.
Motivated by this, we say that 
a topological group has \emph{RP (the Rokhlin property)} if it has a dense conjugacy class.
It has \emph{SRP (the strong Rokhlin property)}  if it has a comeager  conjugacy class. 
 A comeager conjugacy class necessarily has to be a $G_\delta$ (that is, an  intersection of countably many open sets).

Hodges, Hodkinson, Lascar,  and  Shelah \cite{HHLS}, and then Kechris and Rosendal \cite{KR} studied a much stronger notion of ``largeness'' of conjugacy classes.
A topological group $G$ has \emph{$m$-ample generics} if it has SRP in  dimension $m$, that is, if
 the diagonal conjugacy action of $G$ on $G^m$:
\[
 g\cdot(h_1,h_2,\ldots, h_m)=(gh_1g^{-1},gh_2g^{-1},\ldots, gh_mg^{-1})
\]
has a comeager orbit.
It has \emph{ample generics} if it has $m$-ample generics for every $m$.

This last definition was introduced in \cite{KR}. It is slightly different from the definition given in \cite{HHLS}
(see \cite{M}, Chapter 5.2, for more discussion).

We will  call a tuple from this comeager orbit a \emph{generic tuple}.

Groups with ample generics come up naturally in various contexts.
Examples of such groups include:
\begin{enumerate}
\item the group of all automorphisms of the random graph (Hrushovski \cite{H}, see also Hodges et al. \cite{HHLS});
\item the group of all isometries of the rational Urysohn space (Solecki \cite{S});
\item the group of all Haar measure-preserving homeomorphisms of the Cantor set $H(2^\n, \mu)$ (Kechris and Rosendal \cite{KR});
 \item the group of all Lipschitz homeomorphisms of the Baire space $\n^{\n}$ (Kechris and Rosendal \cite{KR}).
\end{enumerate}

Polish groups (separable and completely metrizable topological groups) with ample generics share many  properties 
connecting their algebraic and topological structure. Kechris and Rosendal \cite{KR} showed that if $G$ is a Polish group
that has ample generics, then the  conditions (1)-(3) below hold. See also  \cite{HHLS} for earlier results.
\begin{enumerate}
 \item Every subgroup of $G$ of index less than $2^{\aleph_0}$ is open (small index property).
\item The group $G$ is not a union of  countably many cosets of non-open subgroups
(in particular, $G$   is not a union of a countable  chain of non-open subgroups).\label{dwa}
\item  Every algebraic homomorphism from $G$ to a separable topological group is continuous.
(This  condition implies  that there is exactly one Polish group topology on $G$.)
\end{enumerate}

By a  \emph{permutation group} we mean a closed subgroup of the group of all permutations of natural numbers equipped with the pointwise convergence metric.
It is well known that a  group is a permutation group  exactly when it is the automorphism group  of a countable structure. 
All known examples of groups with ample generics are permutation groups.
A permutation group is \emph{oligomorphic} if it has  finitely many
orbits on each $\mathbb{N}^n$. Equivalently,  it is oligomorphic when it is the automorphism  group of an $\aleph_0$-categorical structure.
Kechris and Rosendal \cite{KR} showed that for an  oligomorphic  group  $G$ with ample generics the following condition holds.

(4) \ 
Whenever $W_0\subseteq W_1\subseteq \ldots \subseteq G=\bigcup_k W_k$, then there are
 $ n$ and $k$ such that $G = W_k^n$. 
 
Condition (4) is called in \cite{KR} the Bergman property. One should point out that the Bergman property is also used in the literature for a weaker property. Condition (4) is called  uncountable strong cofinality in \cite{DG2} (it is stated there in a slightly different, but  equivalent form).

For more background information  on RP, SRP, and ample generics see \cite{KR} or the survey article \cite{GW1}.

Denote the Cantor set by $2^\n$ and the group of homeomorphisms of the Cantor set by $H(2^\n)$.
Akin, Hurley, and Kennedy \cite{AHK} and independently Glasner and Weiss \cite{GW} showed that $H(2^\n)$ has the Rokhlin property.
Later, this result was strengthened by Kechris and Rosendal \cite{KR} who showed that $H(2^\n)$ has the strong Rokhlin property.
Akin, Glasner, and Weiss \cite{AGW} gave a different proof of this result. Moreover, they gave an explicit description of a \emph{generic} homeomorphism of the Cantor set (that is, a homeomorphism with a comeager conjugacy class).

The main result of the paper is the following.
\begin{theorem}\label{verymain}
 The group of homeomorphisms of the Cantor set has ample generics.
\end{theorem}

As  $H(2^\n)$ is an oligomorphic permutation group, as a corollary, we immediately get the following.
\begin{corollary}
 \begin{enumerate}
 \item $H(2^\n)$ has the small index property (Truss \cite{T});
\item $H(2^\n)$ is not a union of  countably many cosets of non-open subgroups;
\item  every algebraic homomorphism from $H(2^\n)$ to a separable topological group is continuous
(Rosendal and Solecki \cite{RS});
 \item
Whenever $W_0\subseteq W_1\subseteq \ldots \subseteq H(2^\n)=\bigcup_k W_k$, then there are
 $ n$ and $k$ such that $H(2^\n) = W_k^n$ (Droste and G\"{o}bel \cite{DG2}).
\end{enumerate}
\end{corollary}

It may be  interesting to compare our results with the results by Hochman \cite{Ho}.
Let $\Gamma$ be a countable discrete group. Let $\mbox{Rep}(\Gamma, H(2^\n))$ be the set of all representations of $\Gamma$ into $H(2^\n)$ (we can also think of it as the
set of all actions of $\Gamma$ on $2^\n$ by homeomorphisms). This is a closed subset of $H(2^\n)^{\Gamma}$.
The group $H(2^\n)$ acts on $\mbox{Rep}(\Gamma, H(2^\n))$ by conjugation. When $\Gamma=F_m$, the free group on $m$ generators, $\mbox{Rep}(\Gamma, H(2^\n))$  can be identified with  $H(2^\n)^m$, and the
action is the diagonal conjugacy action. 
Therefore, saying that $H(2^\n)$ has $m$-ample generics is equivalent to saying that the action of  $F_m$ on  $\mbox{Rep}(F_m, H(2^\n))$ has a comeager orbit.
In contrast,  Hochman \cite{Ho}  showed that all orbits in the action of 
$\mathbb{Z}^m$  ($m>1$) on $\mbox{Rep}(\mathbb{Z}^m, H(2^\n))$ are meager.

The main tool we use in the proof is the \emph{projective Fra\"{i}ss\'{e} theory}
developed by Irwin and Solecki (see \cite{IS}). 
This is a dualization of the Fra\"{i}ss\'{e} theory from model theory.
For each $m$ let $\f_0=\mathcal{F}_0^m$ be the collection of all finite sets $A$ equipped with $m$ directed graphs $s_1^A,s_2^A,\ldots,s_m^A$ such that for every $i$
and every vertex $e\in s_i^A$ there is an edge coming to $e$, and there is an edge going out of $e$.
Maps between members in $\f_0$ are structure preserving surjections.
We show that there is a subfamily $\f=\f^m$ of $\f_0$, which
satisfies the JPP (joint projection property) and the AP (amalgamation property) (Theorem \ref{main})
and is coinitial  in $\f_0$ (Theorem \ref{cap}).
The properties JPP and AP will allow us to take a limit of $\f$ (the projective Fra\"{i}ss\'{e} limit). 
Using the coinitiality of $\f$ in $\f_0$, we show that  this limit is a generic tuple in $H(2^\n)^m$ (Theorem \ref{ample}).

We also present another projective Fra\"{i}ss\'{e}  family, having a simpler description than $\f^1$ considered in the general case,  such that its limit is a generic homeomorphism of the Cantor set.
This will give an alternative proof of the result of Kechris and Rosendal
of the existence of a generic  homeomorphism of the Cantor set. In our proof we identify the class of spiral structures (a~modified
version of finite disjoint unions of finite spirals considered by Akin, Glasner, and Weiss \cite{AGW}) as a projective Fra\"{i}ss\'{e}  family.

The paper is organized as follows.
In Section \ref{sone} we review basic definitions and results on the projective Fra\"{i}ss\'{e} theory.
The proof that $H(2^\n)$ has ample generics is in Section~\ref{sthree}.
In Section~\ref{stwo} we show that the  projective Fra\"{i}ss\'{e} limit of the class of spiral structures
is a generic homeomorphism of the Cantor set.

{\bf{Acknowledgments.}} 
I would like to thank S{\l}awomir Solecki for many suggestions and criticisms that greatly improved the presentation of the paper.
I am also grateful to Kostyantyn Slutskyy for carefully reading a draft of this paper and helpful comments.

\section{Projective Fra\"{i}ss\'{e} theory}\label{sone}
We recall here basic notions and results on the projective Fra\"{i}ss\'{e} theory   developed by Irwin and Solecki in \cite{IS}.

 Given a language $L$ that consists of relation symbols $\{R_i\}_{i\in I}$, and function symbols $\{f_j\}_{\in J}$,
a \emph{topological $L$-structure} is a compact zero-dimensional second-countable space $A$ equipped with
closed relations $R_i^A$ and continuous functions $f_j^A$, $i\in I, j\in J$.
A continuous surjection $\phi\colon B\to A$ is an
 \emph{epimorphism} if it preserves the structure, more precisely, for a function symbol $f$ of arity $n$ and $x_1,\ldots,x_n\in B$ we require:
\[
 f^A(\phi(x_1),\ldots,\phi(x_n))=\phi(f^B(x_1,\ldots,x_n));
\]
and for a relation symbol $R$ of arity $m$ and $x_1,\ldots,x_m\in B$  we require:
\begin{equation*}
\begin{split}
&  (x_1,\ldots,x_m)\in R^A \\ 
&\iff \exists y_1,\ldots,y_m\in B\left(\phi(y_1)=x_1,\ldots,\phi(y_m)=x_m, \mbox{ and } (y_1,\ldots,y_m)\in R^B\right).
\end{split}
\end{equation*}
By an \emph{isomorphism}  we mean a bijective epimorphism.

For the rest of this section fix a language $L$.
Let $\mathcal{F}$ be a family of finite topological\\ $L$-structures. We say that $\mathcal{F}$ is a \emph{ projective Fra\"{i}ss\'{e} family}
if the following two conditions hold:

(F1) (joint projection property: JPP) for any $A,B\in\f$ there are $C\in \f$ and epimorphisms from $C$ onto $A$ and from $C$ onto $B$;

(F2) (amalgamation property: AP) for $A,B_1,B_2\in\f$ and any epimorphisms $\phi_1\colon B_1\to A$ and $\phi_2\colon B_2\to A$, there exist $C$,
 $\phi_3\colon C\to B_1$, and $\phi_4\colon C\to B_2$ such that $\phi_1\circ \phi_3=\phi_2\circ \phi_4$.

A topological $L$-structure $\mathbb{L}$ is a \emph{ projective Fra\"{i}ss\'{e} limit } of $\mathcal{F}$ if the following three conditions hold:

(L1) (projective universality) for any $A\in\f$ there is an epimorphism from $\mathbb{L}$ onto $A$;

(L2) for any finite discrete topological space $X$ and any continuous function
 $f\colon \mathbb{L} \to X$ there are $A\in\f$, an epimorphism $\phi\colon \mathbb{L}\to A$, and a function
$f_0\colon A\to X$ such that $f = f_0\circ \phi$.

(L3) (projective ultrahomogeneity) for any $A\in \f$ and any epimorphisms $\phi_1\colon \mathbb{L}\to A$ and $\phi_2\colon \mathbb{L}\to A$
there exists an isomorphism $\psi\colon \mathbb{L}\to \mathbb{L}$ such that $\phi_2=\phi_1\circ \psi$;

Here is the fundamental result in the projective Fra\"{i}ss\'{e} theory: 
\begin{theorem}[Irwin-Solecki, \cite{IS}]\label{bfraisse}
 Let $\f$ be a countable projective Fra\"{i}ss\'{e} family of finite topological $L$-structures. Then:
\begin{enumerate}
 \item there exists a projective Fra\"{i}ss\'{e} limit of $\f$;\\
\item any two topological $L$-structures that are projective Fra\"{i}ss\'{e} limits are isomorphic.
\end{enumerate}
\end{theorem}

In the propositions below we state some properties of the projective Fra\"{i}ss\'{e} limit.

\begin{proposition}\label{fraisse}
\begin{enumerate}
 \item If $\mathbb{L}$ is the projective Fra\"{i}ss\'{e} limit the following condition (called \emph{the extension property}) holds:
Given $\phi_1\colon B\to A$, $A,B\in\f$, and $\phi_2\colon \mathbb{L}\to A$, then, there is $\psi\colon \mathbb{L}\to B$ such that $\phi_2=\phi_1\circ \psi$.\\
\item If $\mathbb{L}$ satisfies  projective universality, the extension property, and (L2), then it also satisfies  projective ultrahomogeneity, and therefore
is isomorphic to the projective Fra\"{i}ss\'{e} limit.\label{uep}
\end{enumerate}
\end{proposition}

The projective Fra\"{i}ss\'{e} limit is  the inverse limit of certain topological $L$-structures from $\f$.
More precisely, we have the following:

\begin{proposition}\label{inverse}
Let $\f$ be a countable projective Fra\"{i}ss\'{e} family of finite topological $L$-structures. Let $\mathbb{L}$ be its projective Fra\"{i}ss\'{e} limit.
Then, there are $D_1,D_2,D_3,\ldots \in \f$ and $\pi_i\colon D_{i+1}\to D_i$ such that 
$\mathbb{L}$ is the inverse limit of 
\[
\begin{CD}
D_1 @<\pi_1<< D_2 @<\pi_2<< D_3 @<\pi_3<<\ldots,
\end{CD}
\]
and moreover, the following two properties hold:
\begin{enumerate}
 \item For each $A\in\mathcal{F}$ there is $i$ and there is an epimorphism $\phi\colon D_i\to A$.\\
\item For all pairs of epimorphisms $\phi_1\colon B\to A$ and $\phi_2\colon D_i\to A$ there is $j>i$ and $\psi\colon D_j\to B$
 such that $\phi_1\circ \psi =\phi_2\circ \pi^j_i$, where
$\pi^j_i=  \pi_i\circ\ldots\circ \pi_{j-1}$.

\end{enumerate}
\end{proposition}

For more background information on the projective Fra\"{i}ss\'{e} theory and for proofs see \cite{IS} (the proof of Proposition \ref{inverse} is included in the proof of Theorem 2.4 in \cite{IS},
 and the proof of Proposition \ref{fraisse} (ii) goes along the lines of the proof of the uniqueness of the projective Fra\"{i}ss\'{e} limit in \cite{IS}).
For a category-theoretic approach to related issues we refer the reader to \cite{DG1}.

\section{Spiral structures form a projective Fra\"{i}ss\'{e} family}\label{stwo}

The goal of this section is to show that a generic homeomorphism of the Cantor set can be realized as a projective Fra\"{i}ss\'{e} limit
of the class of spiral structures (defined below).
Many ideas in this section are motivated by \cite{AGW}.

\medskip

{\bf{Definition of a spiral structure.}}
Let $R$ be a binary relation symbol.
We define a \emph{spiral} $N=(N,R^N)$ to be the set 
$N=\{1,2,\ldots, n\}$ with two distinguished points $x_N$ and $y_N$ such that $1<x_N<y_N<n$ 
 (we will be referring to them, respectively, as the \emph{left node} of $N$ and the \emph{right node} of $N$), 
equipped with the relation $R^N$ such that
 $R^N(i,i+1)$ for every $i=1,2,\ldots ,n-1$, $R^N(x_N,1)$, and $R^N(n,y_N)$.
See also Figure 1.

\begin{figure}[ht]
\begin{center}
\scalebox{0.2}[0.2]{\includegraphics{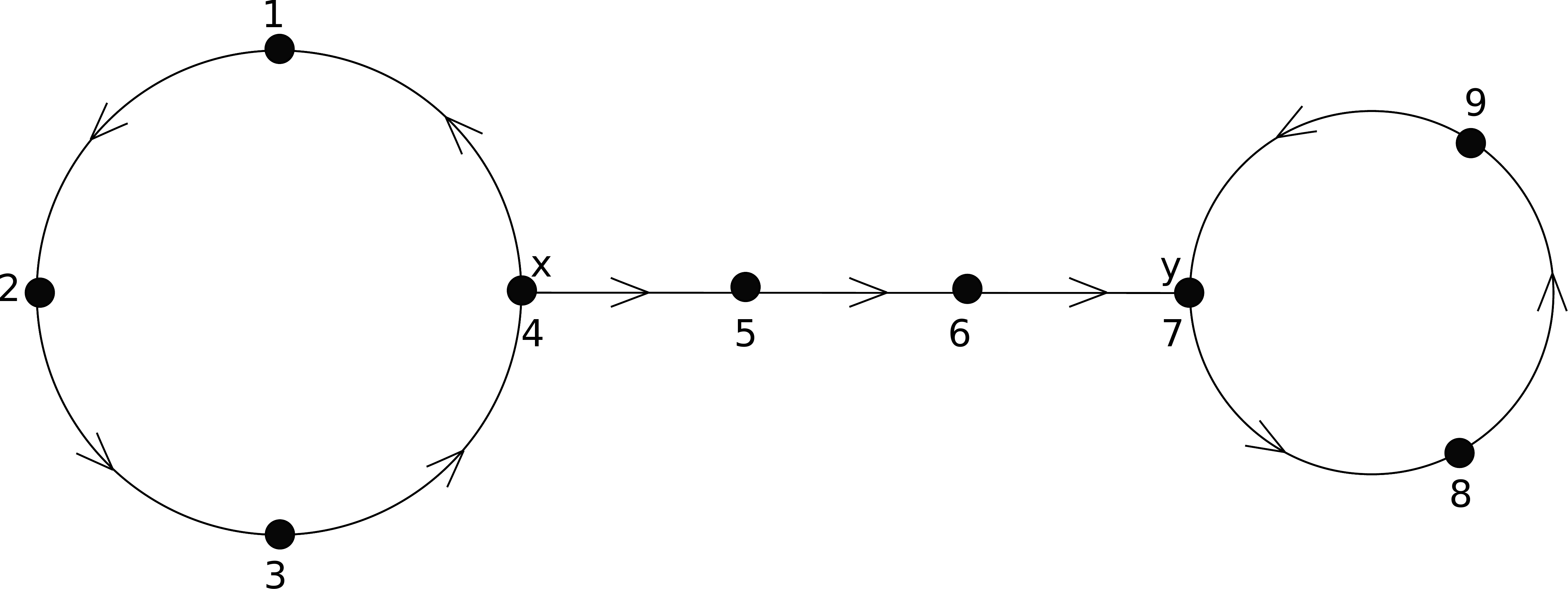}} 
\caption{A spiral}
\label{fig:fig-dots}
\end{center}
\end{figure}

We will call the interval $[1,x_N]$ the {\emph{left circle}} of $N$ and denote it by $l_N$, we
will call the interval $[y_N,n]$ the {\emph{right circle}} of $N$ and denote it by $r_N$, and
we will call the interval $[x_N,y_N]$ the {\emph{middle line}} of $N$ and denote it by $s_N$.
Denote by $|l_N|$ the number of elements in the left circle in $N$, by $|s_N|$ the number of elements  in the middle line of $N$, and
by $|r_N|$ the number of elements in the right circle in $N$.

Spirals come up when we consider a homeomorphism of the Cantor set acting on clopen sets  of the Cantor set.
Take any $f\in H(2^\n)$ and a clopen partition $P$ of $2^\n$.
For $p_0,p_1\in P$ with $f(p_0)\cap p_1\neq \emptyset$ we can choose (usually, in many ways)
a bi-infinite sequence $(p_i)_{i\in \mathbb{Z}}$ with $f(p_i)\cap p_{i+1}\neq \emptyset$, $ {i\in \mathbb{Z}}$, which is eventually periodic as $i\to +\infty$ and $i\to -\infty$; say $\ldots, p_{k-1},p_k$ has period $K$, and $p_l,p_{l+1},\ldots$ has period $L$, where $k<l$. 
Then, we can identify the sequence $p_{k-K+1},\ldots,   p_{k-1},p_k, \ldots p_l,p_{l+1},\ldots p_{l+L-1}$ with a spiral ($p_l$ and $p_k$ become the left and the right node, respectively).
Notice that $f(p_i)\cap p_{i+1}\neq \emptyset$ for every $i=k-K+1,\ldots ,l+L-2$, $f(p_l)\cap p_{l-L+1}\neq \emptyset$, and $f(p_{k+K-1})\cap p_k\neq \emptyset$.

By a \emph{spiral structure} we mean a disjoint union of spirals.
Let $\mathcal{G}$ be the collection of all spiral structures.
The main goal of this section is to show:

\begin{theorem}\label{abcd}
 \begin{enumerate}
  \item The class $\mathcal{G}$ of spiral structures  is a projective Fra\"{i}ss\'{e} family.\\
\item The projective Fra\"{i}ss\'{e} limit of $\mathcal{G}$ is a generic homeomorphism of the Cantor set.
\end{enumerate}
\end{theorem}

{\bf{Maps between spiral structures.}}
We want  to understand epimorphisms between two spiral structures.  
First note that:
\begin{remark}
Let $\phi\colon N\to M$ be an epimorphism between spiral structures. Then, the image of each spiral in $N$ is contained in some spiral of $M$.
Even more, it is either equal to a spiral in $M$, or it is equal to the left circle of a spiral in $M$, or it is equal to the right circle of a spiral in $M$.
\end{remark}
It is therefore enough to describe only relation preserving maps (not necessarily surjective) between spirals. 
Before doing this precisely, let us see a {\emph{typical}}  example of a relation preserving map between spirals.

{\bf{Example.}} 
 Take $M=\{1,2,3,4,5,6\}$ with $x_M=3$ and $y_M=5$.
Take \\$N=\{1,2,3,4,5,6,7,8,9,10\}$ with $x_N=3$ and $y_N=7$.
The map $f\colon N\to M$ satisfying:
$f(1)=2$, $f(2)=3$, $f(3)=1$, $f(4)=2$, $f(5)=3$, $f(6)=4$, $f(7)=5$, $f(8)=6$, $f(9)=5$, and $f(10)=6$
is relation preserving.

In the proposition below we collect information about relation preserving maps between spirals.
\begin{proposition}\label{map}
Let $M=\{1,2,\ldots, m\}$ and $N=\{1,2,\ldots, n\}$ be spirals. Let $f\colon N\to M$  be a  relation preserving map.
Let $x$ be the left node of $M$ and let $y$ be the right node of $M$.

\begin{enumerate}
 \item Suppose that $f$ is onto $M$. Then, there are $a,b\in s_N$ such that $a<b$, $f(a)=x$, $f(b)=y$, and $b-a=|s_M|$
(there is exactly one such a pair $(a,b)$).

Conversely, suppose that $|s_M|\leq|s_N|$, $|l_M|$ divides $|l_N|$ and $|r_M|$ divides $|r_N|$.
Given $a,b\in s_N$ such that $a<b$ and $b-a=|s_M|$, then there is exactly one relation preserving $f\colon N\to M$  that is onto $M$, and such that
$f(a)=x$ and $f(b)=y$.\\ \label{a}
\item Given $f\colon N\to M$ that is onto the left circle of $M$,
then, there is $c\in l_N$ such that $f(c)=x$ (there is more than one such $c$).

Conversely, suppose that  $|l_M|$ divides $|l_N|$ and $|l_M|$ divides $|r_N|$.
Given $c\in l_N$, then there is exactly one relation preserving $f\colon N\to M$  that is onto the left circle of $M$ and satisfies
$f(c)=x$.\\ \label{b}
\item Given $f\colon N\to M$ that is onto the right circle of $M$,
then there is $d\in r_N$ such that $f(d)=y$ (there is more than one such $d$).

Conversely,  suppose that $|r_M|$ divides $|r_N|$ and $|r_M|$ divides $|l_N|$.
Given $d\in r_N$, then there is exactly one relation preserving $f\colon N\to M$  that is onto the right circle of $M$ and satisfies
$f(d)=y$. \label{c}
\end{enumerate}
\end{proposition}

\begin{proof}
In each of \ref{a},\ref{b}, and \ref{c}
 the first statement is immediate, we just use that $f$ is relation preserving.

For the second statement in \ref{a}, we define $f$ in the following way:
$f(b+k)=y+(k \mbox{ mod }(m+1-y))$, for $k=0,1,\ldots, n-b$;
$f(a-k)=x-(k \mbox{ mod } x)$, for $k=0,1,\ldots, a-1$;
$f(k)=x+(k-a)$, for $a\leq k\leq b$.
(Intuitively, everything to the left of $a$ we wrap around the left circle of $M$, and 
everything to the right of $b$ we wrap around the right circle of $M$.)

For the second statement in \ref{b}, we define $f$ in the following way:
$f(c+k)= k \mbox{ mod } x$, for $k=0,1,\ldots, n-c$ (here we identify 0 with $x$);
$f(c-k)= x-(k \mbox{ mod } x)$, for $k=0,1,\ldots, c-1$.

For the second statement in \ref{c}, we define $f$ in the following way:
$f(d+k)=y+(k \mbox{ mod } (m+1-y))$, for $k=0,1,\ldots, n-d$;
$f(d-k)= (m+1)-(k \mbox{ mod } (m+1-y))$, for $k=0,1,\ldots, d-1$ (here we identify $m+1$ with $y$).

\end{proof}

{\bf{Joint projection property.}}
We check that $\mathcal{G}$ has the JPP.
First take two spirals $K$ and $L$. We want to find a spiral $N$ that can be mapped both onto $K$ and onto $L$.
For this, let $N$ be any spiral such that $|l_N|$ divides both $|l_L|$ and $|l_K|$,  $|r_N|$ divides both $|r_L|$ and $|r_K|$, and 
$|s_N|>|s_K|, |s_L|$.
We describe a relation preserving map from $N$ onto $K$: Choose $a,b\in s_N$ with $a<b$ and $b-a=|s_K|$; 
map $a$ to the left node of $K$, map $b$ to the right node of $K$, and extend this to the map on the whole $N$.
We similarly find a relation preserving map from $N$ onto $L$.

In general, when $K$ and $L$ are spiral structures, for every pair of spirals in $K$ and $L$ we find a spiral that can be mapped onto both of them.
The disjoint union of these spirals gives us the required spiral structure.

{\bf{ Amalgamation property.}}
We check that $\mathcal{G}$ has the AP.

The general situation and strategy: We have a spiral structure $K_1\cup\ldots \cup K_n$ (we have here a disjoint union of spirals),
an epimorphism $\phi_1\colon L_1\cup\ldots\cup L_{n_1}\to K_1\cup\ldots \cup K_n$,
and an epimorphism $\phi_2\colon M_1\cup\ldots\cup M_{n_2}\to K_1\cup\ldots \cup K_n$.
Take $L_i$, and consider $\phi_1\restriction L_i$. Its image is contained in some $K_j$. There are three possibilities: the image is equal to $K_j$, or it is equal to the left circle of $K_j$,
or it is equal to the right circle of $K_j$. 

For this fixed $L_i$, take any $M_k$ such that $\phi_2\restriction M_k$ is onto $K_j$.
We find a spiral $N$, a relation preserving map $\phi_3\colon N\to L_i$ that is onto, and a relation preserving map $\phi_4\colon N\to M_k$ (we just want $\phi_4$ to be into) such that $\phi_1\circ \phi_3=\phi_2\circ \phi_4$.
We do this with all of $L_1,L_2,\ldots, L_{n_1}$. Next, we proceed similarly with $M_1,M_2,\ldots, M_{n_2}$.

Therefore, it is enough to show the following:
\begin{proposition}
Let   $K,L,M$ be spirals.
Given a relation preserving map $f_1\colon L\to K$ and a relation preserving map  $f_2\colon M\to K$ that is  onto $K$,
then there exists a spiral $N$, a relation preserving map $f_3\colon N\to L$ that is onto $L$, and 
a relation preserving map $f_4\colon N\to M$
such that $f_1\circ f_3=f_2\circ f_4$.
\end{proposition}

\begin{proof}
Let $x$ and $y$ denote the left and right nodes of $K$, respectively.
 We consider the following three cases.

{\bf{Case 1.}}
The map $f_1$ is onto $K$. Here we will get $f_4$ that is onto $M$.

Take any spiral $N$   such that $|l_N|$ divides both $|l_L|$ and $|l_M|$,  $|r_N|$ divides both $|r_L|$ and $|r_M|$, and 
$|s_N|>3(|s_M|+|s_L|)$.
Take $a_1,b_1\in s_L$ such that $a_1<b_1$, $b_1-a_1=|s_K|$, $f_1(a_1)=x$, and $f_1(b_1)=y$.
Take $a_2,b_2\in s_M$ such that $a_2<b_2$, $b_2-a_2=|s_K|$,  $f_2(a_2)=x$, and $f_2(b_2)=y$.
Choose $a,b\in s_N$ such that $a<b$ and $b-a=|s_K|$. Declare $f_3(a)=a_1$, $f_3(b)=b_1$,  $f_4(a)=a_2$, $f_4(b)=b_2$.
Extend $f_3$ and $f_4$ (in a unique way) to the whole $N$. We do this similarly as in the proof of Proposition \ref{map}  \ref{a}.
Above, we also have to make sure that our chosen $a$ and $b$ satisfy $a_1-x_L,a_2-x_M\leq a-x_N$ and
$y_L-b_1,y_M-b_2\leq y_N-b$.

{\bf{Case 2.}}
The map $f_1$ is onto $l_K$. Here we will get $f_4$ that is onto $l_M$.

Take any spiral $N$ such that   $|l_N|$ divides both $|l_L|$ and $|l_M|$,  $|r_N|$ divides both $|r_L|$ and $|l_M|$, and 
$|s_N|>|l_L|+|s_L|$.
Take $c_1\in l_L$ such that $f_1(c_1)=x$.
Take $c_2\in l_M$ such that $f_2(c_2)=x$.
Choose $c\in l_N$.
Declare $f_3(c)=c_1$ and $f_4(c)=c_2$. Extend $f_3$ (in a non unique way) to the whole $N$ so that $f_3$ is onto $L$.
Extend  $f_4$ (in a unique way) to the whole $N$ so that $f_4$ is onto $l_M$.

{\bf{Case 3.}}
The map $f_1$ is onto $r_K$. Here we will get $f_4$ that is onto $r_M$.

Here we proceed  as in Case 2.
\end{proof}

 Let $(\mathbb{L},R^\mathbb{L})$ denote the projective Fra\"{i}ss\'{e} limit of $\mathcal{G}$.

\begin{proposition}\label{cantor}
The underlying set $\mathbb{L}$ is (homeomorphic to) the Cantor set.
\end{proposition}
\begin{proof}
 The underlying set $\mathbb{L}$  is compact, zero-dimensional, and second-countable, as $(\mathbb{L},R^\mathbb{L})$ is a topological $L$-structure (where $L=\{R\}$).
We  show that $\mathbb{L}$ has no isolated points as follows. Suppose, towards a contradiction, that $p\in \mathbb{L}$ is an isolated point. 
Using (L2) find $A\in\f$ and an epimorphism $\phi\colon \mathbb{L}\to A$ such that the open cover $\{\{p\},\mathbb{L}\setminus\{p\}\}$
is refined by $\{\phi^{-1}(a)\colon a\in A\}$.
Set $a_0=\phi(p)$. We can find $B$ and $\bar{\phi}\colon B\to A$ such that
there are distinct $b_0,b_1$ with $\bar{\phi}(b_0)=\bar{\phi}(b_1)=a_0$
(for example, take $B$ equal to two disjoint copies of $A$, and require $\bar{\phi}$ restricted to each copy to be the identity).
 Using the extension property, find $\psi\colon \mathbb{L}\to B$ such that
$\phi=\bar{\phi}\circ \psi$. Note that $\bar{\phi}^{-1}(b_0)$ and $\bar{\phi}^{-1}(b_1)$ are disjoint non-empty clopen subsets of $\{p\}$.
This gives a contradiction.
\end{proof}

\begin{proposition}\label{bijekcja}
The closed relation $R^\mathbb{L}$ is the graph of a homeomorphism of the Cantor  set.

\end{proposition}

\begin{proof}
Suppose, towards a contradiction, that there are $\alpha,\beta_1,\beta_2\in \mathbb{L}$, $\beta_1\neq \beta_2$, such that $R^\mathbb{L}(\alpha,\beta_1)$ and $R^\mathbb{L}(\alpha,\beta_2)$.
Take $A\in\f$ and $\psi_1\colon \mathbb{L}\to A$ such that $\psi_1(\beta_1)\neq \psi_1(\beta_2)$.
Using the description of epimorphisms between spirals (Proposition \ref{map}) we observe that there are
 $B\in\f$ and $\phi\colon B\to A$ such that whenever $x$ is such that $\phi(x)=\psi_1(\alpha)$, then there is exactly one $y\in B$ such that $R^B(x,y)$.
Using the extension property find $\psi_2\colon \mathbb{L}\to B$ such that $\psi_1=\phi\circ \psi_2$.
We have $R^B(\psi_2(\alpha),\psi_2(\beta_1))$ and $R^B(\psi_2(\alpha),\psi_2(\beta_2))$. 
By the choice of $\phi$,  we get $\psi_2(\beta_1)=\psi_2(\beta_2)$, and therefore $\psi_1(\beta_1)=\psi_1(\beta_2)$.
This gives a contradiction.

We similarly  show that 
there are no $\alpha,\beta_1,\beta_2\in \mathbb{L}$, $\beta_1\neq \beta_2$, such that $R^\mathbb{L}(\beta_1,\alpha)$ and $R^\mathbb{L}(\beta_2,\alpha)$.

As $(\mathbb{L},R^\mathbb{L})$ is a  topological $L$-structure, $R^\mathbb{L}$ is closed and $\mathbb{L}$ is compact, and therefore, the function induced by 
$R^\mathbb{L}$, and its inverse, preserve the topology.
\end{proof}

Denote by $F^\mathbb{L}$ the function induced by $R^\mathbb{L}$. Below, we will be writing
  $(\mathbb{L},F^\mathbb{L})$ rather than $(\mathbb{L},R^\mathbb{L})$.

\begin{proposition}
 The conjugacy class of  $(\mathbb{L},F^\mathbb{L})$ is a dense $G_\delta$ in $H(\mathbb{L})= H(2^{\n})$.
\end{proposition}

\begin{proof}
The proof goes along the lines of proofs of Propositions \ref{dense} and \ref{gdelta}, presented in the next section.
\end{proof}

It is natural to ask whether we can get a generic homeomorphism as a limit of a family of finite sets equipped with just a bijection. In the example below we show
that this is not the case.
Nevertheless, we get a homeomorphism  with a $G_\delta$ conjugacy class.
{\bf{Example.}}
 Let $L=\{F\}$, where $F$ is an unary functional symbol.
Consider
\[
 \f=\{(A,F^A)\colon A \mbox{ is finite }, F^A \mbox{ is a bijection} \}.
\]
This is a projective Fra\"{i}ss\'{e} family. We check JPP and AP.

JPP: Take $(A,F^A),(B,F^B)\in\f$. Then $(A\times B,F^A\times F^B)$ together with projections works.

AP: Take $(A,F^A),(B,F^B),(C,F^C)\in\f$, $\phi_1\colon (B,F^B)\to (A,F^A)$, and $\phi_2\colon (C,F^C)\to (A,F^A)$. Then $(D, F^D)$, where
\[
 D=\{(b,c)\in B\times C\colon \phi_1(b)=\phi_2(c)\}
\]
and $F^D=F^B\times F^C$, together with projections works.

Denote the limit by $(\mathbb{L}, F^\mathbb{L})$. 
Similarly, as in the proof of Proposition \ref{cantor}, we can show that $\mathbb{L}$ is (homeomorphic to) the Cantor set.
Since for every $(A,F^A)\in\f$, $F^A$ is a bijection, it follows that $F^\mathbb{L}$ is a homeomorphism.
The conjugacy class of  $F^\mathbb{L}$ is a $G_\delta$ in $H(\mathbb{L})$. The proof of this goes along the lines of the proof of Proposition \ref{gdelta} presented
in the next section.

{\bf{Claim.}}
 The conjugacy class of $ F^\mathbb{L}$ is not dense in $H(\mathbb{L})$.

\begin{proof}
 Suppose, towards a contradiction, that the conjugacy class of $ F^\mathbb{L}$ is dense in $H(\mathbb{L})$.
For a partition $P=\{p,q\}$ of $\mathbb{L}$ into non-empty clopen sets let
\begin{equation*}
 \begin{split}
  U_P=\{f\in H(\mathbb{L})\colon f(p)\cap p\neq\emptyset,\ f(q)\cap p\neq\emptyset, f(q)\cap q\neq\emptyset, f(p)\cap q=\emptyset\}.
 \end{split}
\end{equation*}
Suppose now that for some $g\in H(\mathbb{L})$
and a partition $P$, $g^{-1}F^\mathbb{L}g\in U_P$. Then $F^\mathbb{L}\in U_{P'}$, where $P'=\{g(p),g(q)\}$.
Using (L2) in the properties of the projective  Fra\"{i}ss\'{e} limit 
(applied to a discrete two-element space $X=\{x,y\}$ and to $f\colon \mathbb{L}\to X$ such that $f^{-1}(x)=g(p)$ and $f^{-1}(y)=g(q)$), it is not difficult to
show that   this is impossible.
\end{proof}

\begin{remark}
 In fact, one can show (for example, by checking conditions (L1), (L2), and (L3) in the definition of the projective  Fra\"{i}ss\'{e} limit) that the limit $(\mathbb{L}, F^\mathbb{L})$
in the example above is isomorphic to $(\Theta\times 2^\n, \tau\times \mbox{id})$, where $(\Theta, \tau)$ is the universal adding machine. 
The universal adding machine is the inverse limit of the inverse system $(\mathbb{Z}_{n!}, p^{n+1}_n)_n$, where $\mathbb{Z}_{n!}$ is the ring of 
integers modulo $n!$, $p^{n+1}_n(k) = k \mbox{ mod } n!$, and $\tau$ is the coordinatewise translation
by the identity element. 
\end{remark}

\section{$H(2^\n)$ has ample generics}\label{sthree}

Let $s$ be a symbol for a binary relation. Following \cite{AHK} (Chapter 8) 
we say that $s^A$ is a surjective relation on a  set $A$ if $s^A\subseteq A^2$ and for any $a\in A$ there are $b,c\in A$
such that $s^A(a,b)$ and $s^A(c,a)$. Note that $s^A$ is a directed graph with an additional surjectivity property.

Surjective relations come up naturally as restrictions of homeomorphisms of the Cantor set  to clopen partitions of the Cantor set.
If $P$ is a clopen partition of $2^\n$ and $f\in H(2^\n)$, then  
$\{(p,q)\in  P^2\colon f(p)\cap q\neq\emptyset\}$
is a surjective relation. 
We can think of a surjective relation as a partial homeomorphism of the Cantor set.
Note also that spiral structures considered in the previous section are surjective relations.

To get a generic $m$-tuple of homeomorphisms, we will consider a certain family $\f$ of $m$-tuples of surjective relations (Theorem \ref{main}).
After taking the limit, we  obtain an $m$-tuple of closed relations on the Cantor set, which are  surjective (that is, projections on both coordinates are onto).
We  show that every relation in this tuple is necessarily a permutation (Proposition \ref{bijekcjaa}), and therefore,  is the  graph of a homeomorphism of the Cantor set.
Finally, we show that this $m$-tuple of homeomorphisms is generic.


Let $L=\{s_1,s_2,\ldots, s_m\}$, where $s_1,s_2,\ldots, s_m$ are symbols for binary relations.
Let
\[
\mathcal{F}_0=\{(A,s^A_1,\ldots,s_m^A)\colon A \mbox{ is a finite non-empty set}, \\
s^A_1,\ldots,s_m^A \mbox{ are surjective relations } \}.
\]

It is straightforward to show that $\mathcal{F}_0$ has the JPP. 
Take $(A,s^A_1,\ldots,s_m^A),(B,s_1^B,\ldots,s_m^B)$ in $\mathcal{F}_0$. Then $(A\times B,s_1^A\times s_1^B,\ldots,s_m^A\times s_m^B)$ together with projections as epimorphisms works.

We want to find a \emph{coinitial} subfamily $\f$ of $\f_0$
(that is, such that for every $A\in \f_0$ there is $B\in\f$ and an epimorphism $\phi\colon B\to A$), which is a projective  Fra\"{i}ss\'{e} family. 
From the coinitiality of $\f_0$ it will follow that $\f$ has the JPP as well. The main difficulty is to take care of the AP.

We start with some notation. 
Let $s_1^{-1},s_2^{-1},\ldots,s_m^{-1}$ be symbols for the inverses of $s_1,s_2,\ldots,s_m$.
For $R$ equal to $s_1,s_1^{-1},\ldots,s_m,s_m^{-1}$, $R^{-1}$ denotes $s_1^{-1}, s_1,\ldots,s_m^{-1}, s_m$, respectively.
Given $A=(A,s^A_1,\ldots,s_m^A)$, then  $(s_1^{-1})^A,\ldots,(s_m^{-1})^A$ are surjective relations too.
Let $R$ be one of $s_1,s_1^{-1},\ldots,s_m,s_m^{-1}$. Given $x\in A$,
we say that $x$ is  $R^A$-\emph{outgoing} if there is more than one $z\in A$ with $R^A(x,z)$, and there is exactly one $y\in A$ with $R^A(y,x)$.
 We say that $x$ is $R^A$-\emph{incoming} if there is more than one $y\in A$ with $R^A(y,x)$, and there is exactly one $z\in A$ with $R^A(x,z)$.
Note that $x$ is $R^A$-outgoing iff it is  $(R^{-1})^A$-incoming.

For $A\in\mathcal{F}_0$ we say that we can \emph{amalgamate over} $A$ if for any $B,C\in\mathcal{F}_0$, $\phi_1\colon B\to A$, and $\phi_2\colon C\to A$ there are $D\in\mathcal{F}_0$,
 $\phi_3\colon D\to B$, and
 $\phi_4\colon D\to C$ such that $\phi_1\circ \phi_3=\phi_2\circ \phi_4$.

Let $\f$  be the collection of all structures from $\f_0$ that satisfy (i) and (ii) of Theorem \ref{main} below. From the coinitiality of $\f$ in $\f_0$ (Theorem \ref{cap} below)
and Theorem \ref{main} it will follow that $\f$ is a  projective  Fra\"{i}ss\'{e} family.
\begin{theorem}\label{main}
 Given $A=(A,s^A_1,\ldots,s_m^A)$, suppose that $A$ satisfies the following conditions.
\begin{enumerate}
\item Every point in $A$ is outgoing  for exactly one of $s_1^A,(s_1^{-1})^A,\ldots,s_m^A,(s_m^{-1})^A$. \label{one}
\item  Let $R$ be one of   $s_1,s_2,\ldots,s_m$. Suppose that $R^A(x,y)$. Then either $x$ is $R^A$-outgoing 
 or $y$ is $R^A$-incoming. \label{two}
\end{enumerate}
Then we can amalgamate over $A$.
\end{theorem}

\begin{remark}
 Condition \ref{two} of Theorem \ref{main} implies that if $R$ is one of  $s_1^{-1},s_2^{-1},\ldots,s_m^{-1}$ and if $R^A(x,y)$, 
then either $x$ is $R^A$-outgoing or $y$ is $R^A$-incoming.
\end{remark}

\begin{proof}[Proof of Theorem \ref{main}]

Given $A=(A,s^A_1,\ldots,s_m^A),B=(B,s_1^B,\ldots,s_m^B),C=(C,s_1^C,\ldots,s_m^C)$, $\phi_1\colon B\to A$, 
$\phi_2\colon C\to A$,
we want to find $D$, $\phi_3\colon D\to B$ and $\phi_4\colon D\to C$ such that $\phi_1\circ \phi_3=\phi_2\circ \phi_4$.

We start with some definitions.
We let
\[
 D_0=\{(b,c)\in B\times C\colon \phi_1(b)=\phi_2(c)\}.
\]
For $R$ equal to one of $s_1,s_1^{-1},\ldots,s_m,s_m^{-1}$ we let
\[
R^{D_0}=\{((b,c), (b',c'))\in D_0\times D_0\colon (b,b')\in R^B, (c,c')\in R^C\}.
\]
Let $\pi_1\colon D_0\to B$ and $\pi_2\colon D_0\to C$ be the projections.
(We will also write $\pi_1,\pi_2$ for restrictions of $\pi_1,\pi_2$ to subsets of $D_0$.)
The surjectivity of $\pi_1$ and $\pi_2$ follows from the surjectivity of $\phi_1$ and $\phi_2$.

The relations $s_1^{D_0},\ldots,s_m^{D_0}$ do not have to be surjective. We  find $D\subseteq D_0$ so that 
$s_1^D=s_1^{D_0}\restriction D,\ldots,s_m^D=s_2^{D_0}\restriction D$ are surjective.
For $n=1,2,3,\ldots$ we let
\begin{equation*}
\begin{split}
 D_n= & \{(x',x'')\in D_{n-1}\colon\mbox{ for every } R = s_1,s_1^{-1},\ldots,s_m,s_m^{-1} \mbox{ there is } \\
& (y',y'')\in D_{n-1} \mbox{ such that } R^{D_0}((x',x''),(y',y''))\}.
\end{split}
\end{equation*}
Let $D=\bigcap_n D_n$. Clearly $s_1^{D_0}\restriction D,\ldots,s_m^{D_0}\restriction D$ are surjective.
We show that $\pi_1\colon D\to B$ and $\pi_2\colon D\to C$ are epimorphisms  (Lemma~\ref{proj}).

Define $E_0=D_0$.
Let $x\in A$. Let $R$ be such that $x$ is $R^A$-outgoing. 
For $n=1,2,3,\ldots$ we define
\begin{equation*}
\begin{split}
 E^x_n=&\{(x',x'')\in E_{n-1}\colon x=\phi_1(x')=\phi_2(x'') \mbox{ and there is }  (y',y'')\in E_{n-1}\\
 &\text{ such that }   R^{D_0}((x',x''),(y',y''))\},
\end{split}
\end{equation*}
and  let
$E_n=\bigcup_{x\in A}E_n^x$.

\begin{lemma}\label{omit}
Let $x\in A$. Let $R$ be such that $x$ is $R^A$-outgoing. Let $n$ be a positive natural number.
Suppose that $(x',x'')\in E_0$ with $\phi_1(x')=\phi_2(x'')=x$, $(y',y'')\in E_{n-1}$, and 
$R^{D_0}((x',x''),(y',y''))$. Then $(x',x'')\in E_{n}$.
\end{lemma}
\begin{proof}
We have $(x',x'')\in E_0$ and  $(y',y'')\in E_i$, for every $i=0,1,\ldots, n-1$.
Furthermore, for every natural number $j$, if $(x',x'')\in E_j$ and $(y',y'')\in E_j$, then $(x',x'')\in E_{j+1}$. 
This gives us $(x',x'')\in E_n$.
\end{proof}

\begin{lemma}
 We have $E_n=D_n$ for every $n$.
\end{lemma}
\begin{proof}
 This is clear for $n=0$.
Suppose it holds for $n$, and we prove it for $n+1$. Clearly $D_{n+1}\subseteq E_{n+1}$. We show $E_{n+1}\subseteq D_{n+1}$.
Take $(x',x'')\in E_{n+1}$. So $(x',x'')\in E_n=D_n$.

First let $R$ 
be such that $x=\phi_1(x')=\phi_2(x'')$ is $R^A$-outgoing.
Then, from the definition of $E^x_{n+1}$, there is $(y',y'')\in E_n=D_n$ such that $R^{D_0}((x',x''),(y',y''))$.

Now let $R$ 
be such that $x$ is not $R^A$-outgoing.
Take $y\in A$ such that $R^A(x,y)$.
Since $x$ is not $R^A$-outgoing, 
$y$ is $R^A$-incoming.
Take any $y'\in B$ and $y''\in C$ such that $R^B(x',y')$ and $R^C(x'',y'')$.
Again, since $x$ is not $R^A$-outgoing,  $y=\phi_1(y')=\phi_2(y'')$, so $(y',y'')\in E_0$. 
From the fact that $y$ is $(R^{-1})^A$-outgoing and $(R^{-1})^{D_0}((y',y''),(x',x''))$,
by Lemma \ref{omit}, we get
$(y',y'')\in E_{n+2}$. Therefore $(y',y'')\in E_n=D_n$. We have proved that $(x',x'')\in D_{n+1}$.
\end{proof}

\begin{lemma}\label{proj}
For every $n=0,1,2,\ldots$:
\begin{enumerate}
 \item[$(i)_n$] $\pi_1[E_n]=B$;\\
\item[$(ii)_n$] for $x',y'\in B$ with $R^B(x',y')$, where $R$ is one of $s_1,s_1^{-1},\ldots,s_m,s_m^{-1}$,
there are $x'',y''\in C$ such that $R^C(x'',y'')$, $\phi_1(x')=\phi_2(x'')$, $\phi_1(y')=\phi_2(y'')$,\\ and $(x',x''),(y',y'')\in E_n$.
\end{enumerate}
\end{lemma}

\begin{proof}[Proof of Lemma \ref{proj}]
{\bf{Proof of $(i)_0$:}}
Clear.

{\bf{Proof that $(ii)_n$ implies $(i)_{n+1}$}}:
Let $x'\in B$ be given. Let $R$ be such that $x=\phi_1(x')$ is $R^A$-outgoing. Take any $y'\in B$ such that $R^B(x',y')$.
Now from $(ii)_n$ we get $x'',y''\in C$ such that
 $(x',x''),(y',y'')\in E_n$ and $R^C(x'',y'') $. 
From the definition of $E_{n+1}^x$ we get $(x',x'')\in E_{n+1}$.

{\bf{Proof that $(i)_n$ implies $(ii)_{n}$}}:
Let $R$ and $x',y'\in B$ with $R^B(x',y')$ be given.
Let $x=\phi_1(x')=\phi_2(y')$. We can assume that $x$ is $R^A$-outgoing. (Otherwise, $y$ is $(R^{-1})^A$-outgoing and the proof is the same.)

If $y$ is $R^A$-incoming, then take any $x'',y''\in C$ with $\phi_2(x'')=x$, $\phi_2(y'')=y$, and $R^C(x'', y'')$. So $R^{D_0}((x',x''),(y',y''))$.
Since $x$ is $R^A$-outgoing and $y$ is $(R^{-1})^A$-outgoing, from Lemma \ref{omit} we get $(x',x''),(y',y'')\in E_n$.

If  $y$ is not $R^A$-incoming, use $(i)_n$ to find $y''\in C$ such that $(y',y'')\in E_n$.
Now take any $x''\in C$ such that $R^C(x'', y'')$. 
Then since $y$ is not $R^A$-incoming, we have
$\phi_2(x'')=x$. 
Note further that since $R^{D_0}((x',x''),(y',y''))$, from the definition of $E_{n+1}^x$, we get $(x',x'')\in E_{n+1}\subseteq E_n$.
This shows $(ii)_{n}$.
\end{proof}


Since there clearly  is $n$ such that $D=E_n$, Lemma \ref{proj} implies that $\pi_1$ is an epimorphism. We similarly show that $\pi_2$ is an epimorphism. Therefore $\phi_3=\pi_1\restriction D$ and $\phi_4=\pi_2\restriction D$ work.

\end{proof}

\begin{theorem}\label{cap}
 The collection of all $B=(B,s_1^B,\ldots,s_m^B)$ satisfying the hypotheses of Theorem \ref{main} is coinitial in $\mathcal{F}$.
\end{theorem}

\begin{proof}
Given  $A=(A,s^A_1,\ldots,s_m^A)$, we take $4m$ disjoint copies of $A$. Call them \\ $A^{+s_i},\widehat{A}^{+s_i},A^{-s_i},\widehat{A}^{-s_i}$,
$i=1,2,\ldots, m$.
Now we define  $B=(B,s_1^B,\ldots,s_m^B)$. 
Let
\[
 B= \bigcup_i\left( A^{+s_i}\cup\widehat{A}^{+s_i}\cup A^{-s_i}\cup\widehat{A}^{-s_i}\right)
\]
be the underlying set.

First some notation. Let $R$ be one of $s_1,\ldots,s_m$. 
For $a\in A$, the copy of $a$ in $A^{+R}$ will be denoted by $a(A^{+R})$, etc.
For $b\in B$, by $p(b)$ we denote the corresponding element in $A$.

Now we define $R^B$.
\begin{enumerate}
 \item For every $(x,y)\in R^A$ we put $(x(A^{+R}),y(A^{-R}))$,  $(x(\widehat{A}^{+R}),y(A^{-R}))$,
 $(x(A^{+R}),\\y(\widehat{A}^{-R}))$, and $(x(\widehat{A}^{+R}),y(\widehat{A}^{-R}))$ into $R^B$.\\
\item For every $b\in B$ choose exactly one $a\in A$ such that $(a,p(b))\in R^A$, and put $(a(A^{+R}),b)$ into $R^B$.\\
\item For every $b\in B$ choose exactly one $a'\in A$ such that $(p(b),a')\in R^A$, and put $(b,a'(A^{-R}))$ into $R^B$.
\end{enumerate}

The relations $s_1^B$ and $s_2^B$ are surjective and
the natural projection from $B$ onto $A$ is an epimorphism. We show that $(B, s_1^B, \ldots, s_m^B)$ is as needed.

{\bf{Claim.}}
The structure $B$ satisfies the hypotheses of Theorem \ref{main}.

\begin{proof}
Let $i=1,2,\ldots,m$.
From the definition of $s_i^B$,
the $s_i^B$-outgoing points are exactly $a(A^{+s_i})$ and $a(\widehat{A}^{+s_i})$, $a\in A$, and 
$(s_i^{-1})^B$-outgoing points
 are exactly $a(A^{-s_i})$ and $a(\widehat{A}^{-s_i})$, $a\in A$.
From this we get \ref{one} of Theorem \ref{main}.
From (i), (ii) and (iii) in the definition of $R^B$ it is clear that  \ref{two} of Theorem \ref{main} is also satisfied.

\end{proof}

\end{proof}
In the rest of this section we show:
\begin{theorem}\label{ample}
The projective Fra\"{i}ss\'{e} limit of $\f$ is a generic tuple in $H(2^\n)^m$.
\end{theorem}

Denote the projective Fra\"{i}ss\'{e} limit of $\f$ by $\mathbb{L}=(\mathbb{L},s_1^\mathbb{L},\ldots,s_m^\mathbb{L})$. 
First we show  that  closed relations $s_1^\mathbb{L},\ldots,s_m^\mathbb{L}$ are graphs of homeomorphisms of the Cantor set,
and then we show that the homeomorphisms induced by $s_1^\mathbb{L},\ldots,s_m^\mathbb{L}$ form a generic tuple, that is, the diagonal conjugacy class of this tuple is comeager.
We borrow some ideas from 
\cite{AGW} (from the proofs of Proposition 3.2  and   Theorem 3.3 in \cite{AGW}).

Let 
\[
\mathcal{G}_0=\{(A,s^A)\colon A \mbox{ is a finite set and } s^A \mbox{ is a surjective relation}\}.
\]
\begin{lemma}\label{init}
 The family $\mathcal{G}$ of spiral structures (defined in Section 2) is coinitial in $\mathcal{G}_0$.
\end{lemma}
\begin{proof}
 Take any $A\in \mathcal{G}_0$. Take $x_0,x_1\in A$ with $R^A(x_0,x_1)$.
Note that the pair $(x_0,x_1)$ can be extended to a bi-infinite sequence $(x_i)_{i\in \mathbb{Z}}$ with $R^A(x_i,x_{i+1})$, $ {i\in \mathbb{Z}}$, which is eventually periodic as $i\to +\infty$ and
$i\to -\infty$. From this we get a spiral $M=M_{(x_0,x_1)}$ and a relation preserving map $f\colon M\to A$ such that for some $x'_0,x'_1\in M$ with $R^M(x'_0,x'_1)$,
$f(x'_0)=x_0$ and $f(x'_1)=x_1$. The required spiral structure is the disjoint union 
\[
\bigcup_{\{(x_0,x_1)\in A^2\colon R^A(x_0,x_1)\}} M_{(x_0,x_1)}.
\]
\end{proof}



\begin{proposition}\label{bijekcjaa}
 The closed relations $s_1^\mathbb{L},\ldots,s_m^\mathbb{L}$ are graphs of homeomorphisms of the Cantor set.
\end{proposition}
\begin{proof}
In Propositions \ref{cantor} and \ref{bijekcja} we showed that the projective Fra\"{i}ss\'{e} limit of $\mathcal{G}_0$ is the graph of a homeomorphism of the Cantor set.
In Lemma \ref{init} we showed that $\mathcal{G}$ is coinitial in $\mathcal{G}_0$.
Let 
\[
\mathcal{G}'=\{(A,s_1^A)\colon  \mbox{ there are } s_2^A,\ldots,s_m^A \mbox{ such that } (A,s^A_1,\ldots,s_m^A)\in\f\}.
\]
This  also is a coinitial in $\mathcal{G}_0$ projective Fra\"{i}ss\'{e} family.

The projective Fra\"{i}ss\'{e} limits of $\mathcal{G}$ and $\mathcal{G}'$ are isomorphic to each other, and they are also isomorphic to
$(\mathbb{L},s_1^\mathbb{L}),(\mathbb{L},s_2^\mathbb{L}),\ldots,(\mathbb{L},s_m^\mathbb{L})$. In particular,  $s_1^\mathbb{L},s_2^\mathbb{L},\ldots, s_m^\mathbb{L}$ are graphs of homeomorphisms of the Cantor set $\mathbb{L}$.
\end{proof}

\begin{remark}
 One can give a more direct  proof of  Proposition \ref{bijekcjaa},  not referring to Section~3. For example, one can adapt the proof of Proposition \ref{bijekcja}
to our situation. 
\end{remark}

We denote the homeomorphisms whose graphs are $s_1^\mathbb{L},\ldots,s_m^\mathbb{L}$ by $F_1^\mathbb{L},\ldots,F_m^\mathbb{L}$, respectively. We  also  write $(\mathbb{L},F_1^\mathbb{L},\ldots,F_m^\mathbb{L})$ rather than $(\mathbb{L},s_1^\mathbb{L},\ldots,s_m^\mathbb{L})$.

By $P$ or $Q$ we  denote partitions of $2^\n$. All partitions will be clopen partitions.
For $f\in H(2^\n)$ and a partition $P$ we define  
\[
f\restriction P=\{(p,q)\in  P^2\colon f(p)\cap q\neq\emptyset\}.
\]
This is a surjective relation. 
Let $(f_1,\ldots,f_m)\restriction P=(f_1\restriction P,\ldots,f_m\restriction P)$.
Define 
\[
 [P,s^P_1,\ldots,s^P_m]=\{(f_1,\ldots,f_m)\in H(2^\n)^m\colon f_1\restriction P=s^P_1,\ldots,f_m\restriction P=s^P_m\}.
\]

\begin{lemma}\label{basis}
Sets of the form $[P,s^P_1,\ldots,s^P_m]$ are clopen in $ H(2^\n)^m$.
Moreover, they form a topological basis in $ H(2^\n)^m$.
\end{lemma}
\begin{proof}
Clearly they are clopen sets.
Take any $(g_1,\ldots,g_m)\in H(2^\n)^m$, and take $\epsilon>0$.  Let
 $U=\{(f_1,\ldots,f_m)\colon \forall i\forall x \ d(f_i(x),g_i(x))<\epsilon\}$
(here $d$ is  any metric on $2^\n$).
This is an open set. We want to find a clopen neighborhood of $(g_1,\ldots,g_m)$ that is of the form $[P,s^P_1,\ldots,s^P_m]$ and is contained in $U$.
For this, take first an arbitrary partition $Q$ of $2^\n$ of mesh $<\epsilon$, and
$P=\{ q_0\cap g_1^{-1}(q_1)\cap\ldots\cap g_m^{-1}(q_m)\colon q_0,q_1,\ldots,q_m\in Q\}$.
For $i=1,2,\ldots, m$, we let $s_i^P=\{(p,r)\colon g_i(p)\cap r\neq\emptyset\}$.
Clearly $(g_1,\ldots,g_m)\in [P,s^P_1,\ldots,s^P_m]$.
Now take any $(f_1,\ldots,f_m)\in [P,s^P_1,\ldots,s^P_m]$,
and $p\in P$, say $p= q_0\cap g_1^{-1}(q_1)\cap\ldots\cap g_m^{-1}(q_m)$. Then $g_i(p)\subseteq q_i$ for every $i=1,2,\ldots, m$.
 For any $r\in P$, $f_i(p)\cap r\neq\emptyset$ iff $g_i(p)\cap r\neq\emptyset$ ($i=1,2,\ldots, m$).
Therefore $f_i(p)\subseteq q_i$, $i=1,2,\ldots, m$. Since $\mbox{diam}(q_i)<\epsilon$, 
 for every $i=1,2,\ldots, m$ and $x\in p$, $d(f_i(x),g_i(x))<\epsilon$. Since $p\in P$ was arbitrary, this shows
$(f_1,\ldots,f_m)\in U$.
\end{proof}

\begin{proposition}\label{dense}
 The conjugacy class of  $(F_1^\mathbb{L},\ldots,F_m^\mathbb{L})$ is dense in $H(\mathbb{L})^m= H(2^\n)^m$.
\end{proposition}

\begin{proof}
For a partition $P$ and a tuple of surjective relations $(s^P_1,\ldots,s^P_m)$ on $P$ we consider
\[
 D(P,s^P_1,\ldots,s^P_m)=  \{(f_1,\ldots, f_m)\in H(\mathbb{L})^m\colon 
\exists g\ (g^{-1}f_1 g,\ldots,g^{-1}f_m g)\in [P,s^P_1,\ldots,s^P_m]\}.
\]
Let $ D$ be the intersection of all sets of the form $ D(P,s^P_1,\ldots,s^P_m)$.
From Lemma \ref{basis} it follows that if $(f_1,\ldots,f_m)\in D$, then it  has a dense conjugacy class.

We show that $(F^\mathbb{L}_1,\ldots,F^\mathbb{L}_m)\in D$. 
Fix a partition $P$ and a tuple $(s^P_1,\ldots,s^P_m)$ of surjective relations  on $P$.
From the projective universality of the limit and the coinitiality of $\f$ in $\f_0$, there are a partition $Q$ 
and an isomorphism  $i\colon (P,s^P_1,\ldots,s^P_m)\to (Q,F^\mathbb{L}_1\restriction Q,\ldots,F^\mathbb{L}_m\restriction Q)$.
Now take any $g\in H(\mathbb{L})$ that  extends  $i$, and notice that $(g^{-1}F_1^\mathbb{L} g,\ldots,g^{-1}F_m^\mathbb{L} g)\in [P,s^P_1,\ldots,s^P_m]$.
\end{proof}

\begin{proposition}\label{gdelta}
 The conjugacy class of  $(F_1^\mathbb{L},\ldots,F_m^\mathbb{L})$ is a $G_\delta$ in $H(\mathbb{L})^m= H(2^\n)^m$.
\end{proposition}

\begin{proof}
We show that the set of $(f_1,\ldots,f_m)\in H(2^\n)^m$ such that $(2^\n,f_1,\ldots,f_m)$ satisfies (L1), the extension property, and (L2), is a $G_\delta$.
From Proposition \ref{fraisse} (ii), these are exactly structures  that are isomorphic to the projective Fra\"{i}ss\'{e} limit $(\mathbb{L},F_1^\mathbb{L},\ldots,F_m^\mathbb{L})$,
that is, structures that are conjugate to $(\mathbb{L},F_1^\mathbb{L},\ldots,F_m^\mathbb{L})$.

1. Given $A\in\mathcal{F}$, we notice that 
\[
 U_A=  \{(f_1,\ldots,f_m)\in H(2^\n)^m\colon \mbox{ there is an epimorphism from } 
 (2^\n,f_1,\ldots,f_m) \mbox{ onto } A\}
 \]
is open.

2. Given $A=(A_0,s_1^{A},\ldots,s_m^{A}),B=(B_0,s_1^{B},\ldots,s_m^{B})\in \mathcal{F}$, $\phi\colon B\to A$, and a continuous surjection $g\colon 2^\n\to A_0$,
consider 
\begin{equation*}
\begin{split}
E_{\phi,g}=&\{(f_1,\ldots,f_m)\in H(2^\n)^m\colon \mbox{ if } g\colon (2^\n,f_1,\ldots,f_m)\to A \mbox{ is an epimorphism,}  \\
& \mbox{then there is } h\colon (2^\n,f_1,\ldots,f_m)\to B \mbox{ such that }  g=\phi\circ h\}.
\end{split}
\end{equation*}
We show that this set is open.

For $A$ and $g\colon 2^\n\to A_0$ as above we define
\[
 H(g,A)=\{(f_1,\ldots,f_m)\in H(2^\n)^m\colon g\colon(2^\n,f_1,\ldots,f_m)\to A \mbox{ is an epimorphism} \}.
\]
This is a clopen set in $H(2^\n)^m$. Therefore
\begin{equation*}
\begin{split}
E_{\phi,g}=  \left(H(2^\n)^m\setminus H(g,A)\right)\cup \left( \bigcup_{h} H(h,B) \right),
\end{split}
\end{equation*}
where the union is taken over  continuous surjections $h\colon 2^\n\to B_0$ such that $g=\phi\circ h$, is an open set.
Since there are only countably many clopen decompositions of $2^\n$, there are only countably many 
continuous surjections $g\colon 2^\n\to A_0$.

3. Clearly, every  $(2^\n, f_1,\ldots,f_m)$ satisfies (L2).

Hence, 
\[
\left(\bigcap_A U_A \right)\cap \left(\bigcap_{\phi,g} E_{\phi,g} \right)
\]
is a $G_\delta$ set.
It consists exactly of $(f_1,\ldots,f_m)\in H(2^\n)^m$ such that $(2^\n, f_1,\ldots,f_m)$ satisfies (L1), the extension property, and (L2).

\end{proof}

\begin{proof}[Proof of Theorem \ref{verymain}]
 This follows from  Theorems \ref{main}, \ref{cap}, and \ref{ample}.
\end{proof}

\address{\noindent Department of Mathematics, University~of~Illinois~at~Urbana-Champaign, \\1409 W. Green St.,
 Urbana, IL 61801 }

   \email{akwiatk2@illinois.edu}

\end{document}